\documentclass{aims} 
\usepackage{amsmath}
\usepackage{paralist}
\usepackage[misc]{ifsym}
\usepackage{epsfig} 
\usepackage{epstopdf} 
\usepackage[colorlinks=true]{hyperref}
\hypersetup{urlcolor=blue, citecolor=red}
\allowdisplaybreaks

\def\currentvolume{X}
 \def\currentissue{X}
  \def\currentyear{200X}
   \def\currentmonth{XX}
    \def\ppages{X-XX}
     \def\DOI{10.3934/xx.xxxxxxx}

\usepackage[capitalise,nameinlink]{cleveref}
\usepackage{enumitem} 
\usepackage{amsfonts,amssymb,amsthm,amsopn}
\usepackage[centercolon]{mathtools}
\usepackage{dsfont}

\usepackage{algorithm}
\usepackage{algorithmic}

\DeclareMathOperator{\sgn}{sgn}
\newcommand{\e}{\mathrm{e}}
\newcommand{\dd}{\mathrm{d}}

\newcommand{\Uad}{U_{\textup{ad}}}
\newcommand{\N}{\mathbb{N}}
\newcommand{\R}{\mathbb{R}}
\newcommand{\C}{\mathbb{C}}

\newtheorem{theorem}{Theorem}[section]

\newtheorem{lemma}[theorem]{Lemma}
\newtheorem{proposition}[theorem]{Proposition}

\theoremstyle{definition}

\newtheorem{remark}[theorem]{Remark}

\newtheorem{Algorithm}[theorem]{Algorithm}
\crefname{Algorithm}{Algorithm}{Algorithms}
\crefname{ALC@unique}{Line}{Lines}

\title[A control with infinitely many switches]
{A bang-bang solution with infinitely many switching points for a parabolic boundary control problem with terminal observation} 

\author[Constantin Christof]{}


\thanks{$^*$Corresponding author: Constantin Christof}

\begin{document}
\maketitle

\centerline{\scshape
Constantin Christof$^{{\href{mailto:christof@mathematik.tu-darmstadt.de}{\textrm{\Letter}}}*}$}

\medskip

{\footnotesize
 \centerline{TU Darmstadt,
Department of Mathematics,}
\centerline{Dolivostra{\ss}e 15,
64293 Darmstadt, Germany}
} 

\bigskip

 \centerline{(Communicated by Handling Editor)}


\begin{abstract}
We study a parabolic boundary control problem with one spatial dimension, control constraints of box type, and an objective function that measures the $L^2$-distance to a desired terminal state. It is shown that, for a certain choice of the desired state, the considered problem possesses an optimal control that is \emph{chattering}, i.e., of bang-bang type with infinitely many switching points and a positive objective function value. Whether such a solution is possible has been an open question in the literature. We are able to answer this question in the affirmative by means of Fourier analysis and an auxiliary result on the existence of power series with a certain structure and sign-changing behavior. The latter may also be of independent interest.
\end{abstract}



\section{Introduction} 
\label{sec:1}
This paper is concerned with optimal control problems of the following type:
\begin{equation}
\label{eq:P}
\tag{P}
\left.
\begin{aligned}
\text{Minimize } & \frac12 \left \| y(T, \cdot) - y_d \right \|_{L^2(0,1)}^2
\\
\text{w.r.t.\ } &y \in L^2(0,T;H^1(0,1)) \cap H^1(0,T;H^1(0,1)^*),~~u \in L^2(0,T),
\\
\text{ s.t.\ } & 
\left \{
\begin{aligned}
		\partial_t y - \partial_x^2 y &= 0 &&\hspace{-0.15cm}\text{in } (0,T) \times (0,1), \\
		\partial_x y &= 0 &&\hspace{-0.15cm}\text{on } (0,T) \times \{0\}, \\
		\partial_x y &= u &&\hspace{-0.15cm}\text{on } (0,T) \times \{1\}, \\
		y &= 0  &&\hspace{-0.15cm}\text{on } \{0\} \times (0,1), \\
		\end{aligned}
		\right .
\\
\text{ and } & u \in \Uad := \{ v \in L^2(0,T) \mid  -1 \leq v \leq 1 \text{ a.e.\ in } (0,T)\}.
\end{aligned}
~~\right \}
\end{equation}
Here, $y_d \in L^2(0,1)$ is a given desired terminal state;
$T>0$ is the final observation time;
the appearing Sobolev, Lebesgue, and Bochner spaces
are defined as usual;
and the temporal and spatial derivatives $\partial_t$, $\partial_x$, and $\partial_x^2$
are understood in the weak sense;
see \cref{sec:2}.
The main purpose of this work is to construct a target state $y_d$
for which \eqref{eq:P} possesses an optimal control $\bar u$ that
 is \emph{chattering}, 
i.e., of bang-bang type with infinitely many switching points and a positive objective function value.
The problem of deciding whether such an optimal control is possible 
for \eqref{eq:P} and of constructing a corresponding example
has been referred to as difficult and open in the literature;
see \cite[pages 290 and 297]{troeltzsch2023bang}
and \cite[Remark 3.4]{Glashoff1975}. 
With the present paper, we fill this gap in the theory of bang-bang optimal control problems. 
For the main result of our analysis, we refer the reader to \cref{th:main_bang}. 

Before we begin with our analysis, let us put our work into perspective:
Optimal control problems of the type \eqref{eq:P} have been 
studied (at least) since the nineteen-sixties and -seventies; see the seminal works of Fuller 
\cite{Fuller1963} and Glashoff \cite{Glashoff_Krabs1975,Glashoff1975}. 
Their main feature is that they promote 
optimal controls $\bar u$ which satisfy 
$\bar u(t) \in \{-1,1\}$ for a.a.\ $t \in (0,T)$ and, thus, 
can be realized 
comparatively easily in engineering applications. The latter  is in particular true
if it can be shown that the solutions do not oscillate too erratically
between the two available function values $-1$ and $1$. 
As a consequence, the study of the switching behavior of optimal controls of problems like \eqref{eq:P}
has been the subject of extensive research in the past.
Compare, e.g., with \cite{Dhamo2011,Fuller1963,Glashoff_Krabs1975,Glashoff1975,Glashoff1979,Glashoff1980,Karafiat1977,Qin2021,troeltzsch2023bang,TroeltzschWachsmuth2018,White1984} in this context.

A main result in this research area is the 
so-called \emph{bang-bang principle}; see \cite{Glashoff_Krabs1975,Glashoff1975}, 
\cite[Section 3.2.4]{Troeltzsch2010}, and \cref{prop:bang_bang_principle} below. 
This principle implies that,
if the optimal objective function value of \eqref{eq:P} is positive
(i.e., if there is no admissible state $y$ satisfying $y_d = y(T,\cdot)$ a.e.\ in $(0,1)$),
then \eqref{eq:P} admits precisely one optimal control $\bar u$
and
this control possesses a (necessarily unique) right-continuous representative
$\bar u \colon [0,T) \to \R$ 
which satisfies $\bar u(t) \in \{-1,1\}$ for all $t \in [0,T)$ and changes its function value
at most a finite number of times on each of the intervals $[0, T - \varepsilon]$, 
$\varepsilon > 0$.
Note that the latter implies in particular that 
$\bar u \colon [0,T) \to \R$  has bounded variation 
on $[0, T - \varepsilon]$ for all $\varepsilon > 0$ and
is discontinuous on $[0,T)$ at at most countably many 
points---the so-called switching points---which can accumulate at $T$ only.  

A natural question that arises with regard to the bang-bang principle 
is whether it is indeed possible that 
$\bar u \colon [0,T) \to \R$ switches its function value infinitely often on $[0,T)$.
Solutions with such an infinite number of switching points (and a positive objective function value) 
are often 
called \emph{chattering} (see \cite[Section~2.3.4]{Sager2012})
and are of particular interest because they provide benchmark problems 
for numerical algorithms that are provably hard;
cf.\ \cite[Section~5]{Sager2012}. Moreover, if it can be shown
that chattering solutions do not exist,
this has far-reaching consequences for the derivation 
of finite element error estimates and quadratic growth conditions
as it implies, for example, that optimal controls possess $BV([0,T])$-regularity;
cf.\ \cite{Christof2018,Casas2017,Muench2013}. 
We remark that the absence of chattering solutions 
can indeed be established
for problems of the type \eqref{eq:P} if the
objective function has the form $\|y(T, \cdot) - y_d\|_{L^\infty(0,1)}$, $y_d \in L^\infty(0,1)$;
see \cite{Glashoff1979,Glashoff1980,Karafiat1977}. The question of whether a similar effect is 
present for the problem \eqref{eq:P} itself with its terminal $L^2(0,1)$-tracking term 
or if, for this kind of problem, chattering can occur
has---to the best of the author's knowledge---not yet been answered in the literature.
We again refer to \cite[pages 290 and 297]{troeltzsch2023bang}
and \cite[Remark 3.4]{Glashoff1975} for comments on this topic. 

The main difficulty that arises when one tries 
to construct an instance of problem
\eqref{eq:P} which is subject to the bang-bang principle and possesses a chattering solution 
is that the necessary and sufficient optimality conditions of \eqref{eq:P}
imply that the switching structure of an optimal control $\bar u$ with associated
state $\bar y$ 
is dictated by the sign of the trace $\bar p (\cdot, 1)\colon [0,T) \to \R$
of the solution $\bar p$ of the adjoint equation
\begin{equation}
\label{eq:adjoint}
\begin{aligned}
		-\partial_t \bar p - \partial_x^2 \bar p  &= 0 &&\hspace{-0.15cm}\text{in } (0,T) \times (0,1), \\
		\partial_x \bar p  &= 0 &&\hspace{-0.15cm}\text{on } (0,T) \times \{0 \}, \\
		\partial_x \bar p  &= 0 &&\hspace{-0.15cm}\text{on } (0,T) \times \{ 1\}, \\
		\bar p &= \bar y- y_d   &&\hspace{-0.15cm}\text{on } \{T\} \times (0,1).
\end{aligned}
\end{equation} 
More precisely, for the case of a positive optimal objective function value, 
it can be shown that  $\bar u(t) = - \sgn(\bar p(t, 1))$ holds for a.a.\ $t \in (0,T)$;
see \cref{prop:bang_bang_principle} below. This entails that, 
in order to construct a chattering optimal control $\bar u$, one has to find a terminal datum 
for the system \eqref{eq:adjoint} that gives rise to a function $\bar p$
which is highly oscillatory on $[0,T) \times \{1\}$---a task that is 
nontrivial given the mollifying effects of the heat equation and the smoothness 
properties of caloric functions. Additionally, one, of course,
has to close the system so that the resulting $\bar p$
induces an optimal control $\bar u = - \sgn(\bar p(\cdot, 1))$
whose associated state $\bar y$ fits to the terminal datum in \eqref{eq:adjoint} that one started with. 
The strategy that we employ to overcome these difficulties 
can be summarized as follows:
\begin{enumerate}[label=\roman*)]
\item\label{strat:item:i}
Using Fourier analysis, we reduce the problem of finding a suitable 
terminal datum for \eqref{eq:adjoint} to that of constructing a real power series 
that has radius of convergence one,
involves only certain exponents $\alpha_m \in \N$,
has a coefficient sequence that is $q$-summable for all $q \in (1, \infty)$,
and changes its sign infinitely many times on the interval $(0,1)$; see question \eqref{eq:Q} below.
\item\label{strat:item:ii}
 We establish the existence of a power series with the structural properties and sign-changing behavior 
in step \ref{strat:item:i}
by means of an explicit algorithmic construction that relies on properties of 
the harmonic and geometric series. 
\item\label{strat:item:iii}
 With the power series from step \ref{strat:item:ii} and the associated terminal datum for \eqref{eq:adjoint} 
at hand---$w \in L^2(0,1)$  let's say---we solve  \eqref{eq:adjoint}, define $\bar u := - \sgn(\bar p(\cdot,1))$,
and solve the governing equation of \eqref{eq:P} 
for the associated state $\bar y$. 
\item\label{strat:item:iv}
 We define $y_d := \bar y(T, \cdot) - w$, note that, with this definition,
the necessary optimality conditions of \eqref{eq:P} are satisfied, 
and use the convexity of the problem to conclude that we have indeed 
found a chattering optimal control. 
\end{enumerate}

We remark that the Fourier approach in step \ref{strat:item:i} as well as 
the arguments in 
steps \ref{strat:item:iii} and \ref{strat:item:iv} above follow very closely 
the ideas of \cite{troeltzsch2023bang}
where the construction of bang-bang solutions 
with a finite number of switching points at prescribed 
locations is discussed. In this paper, Fourier analysis is used to reduce 
the problem of finding a suitable terminal datum for \eqref{eq:adjoint} 
to that of constructing a polynomial with certain properties. 
The existence of this polynomial is then established by means of a 
finite-dimensional linear system and results on 
generalized Vandermonde determinants. 
Note that, due to its reliance on finite-dimensional arguments, 
this approach cannot be extended
to the construction of chattering optimal controls. In fact, 
the main difficulty in constructing such a control for \eqref{eq:P} is
establishing that a power series with the properties in step \ref{strat:item:i}
above can indeed be found. The existence result that 
we prove for this purpose---\cref{th:power_series}---can be seen 
as the second main result of this work. Lastly, we would like to point out that
our construction does not allow 
to prescribe where $\bar u$ changes its function value. Our approach only yields information about
how often changes of the function value occur but no precise 
knowledge about where the switching points lie. This is a main difference 
to the results of \cite{troeltzsch2023bang}.

We conclude this introduction with a brief overview of the content and structure of the remainder 
of the paper. 

In \cref{sec:2}, we clarify the notation and introduce basic concepts. 
\Cref{sec:3} is then concerned with the question of how a power series 
suitable for the construction of a chattering solution of \eqref{eq:P} can be found. 
As the results of this section are also of interest for other applications, 
we present them in a manner independent of the optimal control context of \eqref{eq:P}.
In \cref{sec:4}, we finally combine the results of 
\cref{sec:3} with the Fourier approach outlined above to arrive 
at an instance of \eqref{eq:P} that indeed gives rise to a chattering 
optimal control. 

\section{Basic notation}
\label{sec:2}

In this paper, we denote 
norms and inner products defined on a real vector space $X$
by $\|\cdot\|_X$ and $(\cdot, \cdot)_X$, respectively. 
For the topological dual space of a normed space $X$,
we use the notation $X^*$, and for convergence in a normed space, the symbol $\to$. 
The natural numbers,  the reals, and the complex numbers
are denoted by the standard symbols $\N$,  $\R$, and $\C$, respectively.
For the absolute value function on $\R$ or $\C$, we write $|\cdot|$.
If $1 \leq q < \infty$ is given, then  $\ell^q$
denotes the vector space of $q$-summable real sequences, endowed with its canonical norm $\|\cdot\|_{\ell^q}$;
see \cite[Section~2.23]{Alt2016}. Given a nonempty and open set 
$\Omega \subset \R^d$, $d \in \N$, we further use the standard notation 
$L^q(\Omega)$, $W^{k,q}(\Omega)$, and $H^k(\Omega)$, $1 \leq q \leq \infty$, $k \in \N$,
for the Lebesgue and Sobolev spaces on $\Omega$, again equipped with their 
usual norms; see \cite[Chapter 5]{Evans2010}. 
The space of  functions $v\colon [a,b] \to \R$ of bounded variation on an interval $[a,b]$,
$-\infty < a< b < \infty$, is denoted by $BV([a,b])$. If $T>0$ is given and $X$ is a Banach space, then we 
use the symbols $L^q(0, T; X)$, $W^{k, q}(0, T;X)$, and $H^k(0, T;X)$, $1 \leq q \leq \infty$, $k \in \mathbb{N}$,
to denote the Lebesgue-Bochner and the Sobolev-Bochner spaces on $(0,T)$, as defined in \cite[Section 5.9.2]{Evans2010}. 
For derivatives w.r.t.\ a time variable $t$, we write $\partial_t$, and for derivatives 
w.r.t.\ a spatial variable $x$, we write $\partial_x$. For elements of the Sobolev and Sobolev-Bochner spaces, 
derivatives are understood in the weak sense, as usual. 
When talking, e.g., about the evaluation
of an element of $H^1(0, T;X)$ at $t =0$ or about the evaluation of 
a function $v \in H^1(0,1)$ at $x \in \{0,1\}$, this is always understood in the sense 
of temporal/spatial traces. 

\section{An auxiliary result on power series}
\label{sec:3}

As we will see in \cref{sec:4}, to construct an instance of problem \eqref{eq:P}
with a chattering optimal control, we have to deal with the following question:
\smallskip
\begin{equation}
  \tag{Q}\label{eq:Q}
  \left.
  \parbox{\dimexpr\linewidth-6em}{%
    \strut
    \emph{%
    Given a (not necessarily increasing) sequence 
    $\{\alpha_m\}_{m \in \mathbb{N}} \subset \mathbb{N}$
    satisfying $\alpha_m \geq m$ for all $m \in \mathbb{N}$,
    is it always possible to find a sequence \mbox{$\{\beta_m\}_{m \in \N} \subset \R$}  such that:
   \begin{enumerate}[label=\roman*)]
    \item  $\{\beta_m\}_{m \in \N} \in \ell^q$ holds for all $q \in (1, \infty)$;
    \item  the power series 
    	\[
		P(z) := \sum_{m=1}^\infty \beta_m z^{\alpha_m}
	\]
	has radius of convergence one and changes its sign infinitely many times on $(0,1)$?
    \end{enumerate}
    }
  }
  ~~
  \right \}
\end{equation}
\smallskip

Note that answering the above question is indeed nontrivial.
In the case $\alpha_m = m^2$ for all $m \in \N$, for example, 
one has $\{\alpha_m^{-1}\}_{m \in \N} \in \ell^1$ in \eqref{eq:Q},
which entails by the theorem of M{\"u}ntz-Sz{\'a}sz that the approximation capabilities 
of the monomials $z^{\alpha_m}$ on $(0,1)$ are very limited; see 
\cite[Section 4.2]{Borwein1995}. Moreover, the summability 
condition $\{\beta_m\}_{m \in \N} \in \ell^q$ for all $q \in (1, \infty)$ implies that 
the coefficient sequence  
$\{\beta_m\}_{m \in \N}$ has to decay rather rapidly for $m \to \infty$.
The main goal of this section is to answer the question \eqref{eq:Q} in the affirmative. 
To this end, we set up an algorithm that, 
given a sequence of exponents $\{\alpha_m\}_{m \in \mathbb{N}} \subset \mathbb{N}$ 
satisfying $\alpha_m \geq m$ for all $m \in \N$,
produces a sequence $\{\beta_m\}_{m \in \N} \subset \R$ such that 
the two conditions in \eqref{eq:Q} are met. 
We remark that this constructive procedure also allows 
to determine the coefficients $\beta_m$ numerically and, thus, 
to explicitly compute examples of desired states $y_d$ that give rise to chattering 
controls in \cref{sec:4}; see \Cref{fig:1,fig:2}  and \cref{tab:1} below.
As the results of this section may also be of interest for other 
applications, we formulate them independently of the optimal control context of \eqref{eq:P}.

The main idea of our construction is to 
pit the convergence of the geometric series against the divergence 
of the  harmonic one. 
More precisely, we set up our algorithm such that 
it produces a sequence of coefficients 
 $\{\beta_m\}_{m \in \N} \subset \R$ that 
 is of the following form for suitable 
 numbers $r_1 := 1 < r_2 < r_3 < \ldots$ with $r_k \in \N$:
 \begin{equation}
 \label{eq:bmstructure}
 1 , 0, \ldots,  0, 
 -\frac{1}{2}, -\frac{1}{3}, \ldots, - \frac{1}{r_2},
 0, \ldots, 0, 
 \frac{1}{r_2 + 1},\ldots, \frac{1}{r_3},
 0, \ldots, 0, 
 - \frac{1}{r_3+1}, \ldots
 \end{equation}
That is, $\{\beta_m\}_{m \in \N} $ is
a harmonic sequence which is split into blocks that are separated by zeros
and endowed with alternating signs. In what follows, 
$r_k \in \N$ always stands for
the absolute value of the 
reciprocal of the last number in the $k$-th nonzero block in this structure.
We further denote by 
$p_k \in  \mathbb{N}$ the sequence index of 
$\{\beta_m\}_{m \in \N} $ at which the $k$-th nonzero 
block starts and 
by $q_k \in \N$ the sequence index of $\{\beta_m\}_{m \in \N}$
at which the $k$-th nonzero block ends. 
Note that, in \eqref{eq:bmstructure}, this means that we have 
$p_1 = q_1 = r_1 = 1$, $p_2 =  N_1 +2$, and 
$q_2 = r_2 + N_1$,
where $N_1$ is the number of zeros in the first zero block. 

The main difficulty that arises when following the above lines 
is to set up the block structure in \eqref{eq:bmstructure} such that 
the power series $P$ in \eqref{eq:Q} indeed changes its sign infinitely often on $(0,1)$,
even in the presence of rapidly increasing exponents $\alpha_m$. 
The method that we use to accomplish this can be seen in \cref{algo:AlgoHarmonic}.
\vspace{0.2cm}

\begin{Algorithm}[Algorithmic solution of \eqref{eq:Q}]\label{algo:AlgoHarmonic}~
\begin{algorithmic}[1]
  \STATE Input:  $z_1 \in (0,1)$ and $\{\alpha_m\}_{m \in \mathbb{N}} \subset \mathbb{N}$ satisfying $\alpha_m \geq m$ for all $m \in \mathbb{N}$.
  \STATE\label{algo:step:2} Define $p_1 := 1$, $q_1 := 1$, $r_1 := 1$, $\beta_1 := 1$.
    \FOR{$k=1,2,3,4,\ldots$}  \label{algo:step:3}
    \STATE\label{algo:step:4}  Choose $p_{k+1} \in \N$ such that $p_{k+1} > q_k$ holds and 
\begin{equation}
\label{eq:randomeq363738}
\left |
\sum_{m=1}^{q_k} \beta_m  z_k^{\alpha_m} 
\right |
> 
 \sum_{m=p_{k+1}}^{\infty}   z_k^{\alpha_m} 
 > 0.
\end{equation}
\STATE\label{algo:step:5} Set $\beta_m := 0$ for all $m= q_k + 1,\ldots , p_{k+1} - 1$. 
\STATE\label{algo:step:6} Choose $r_{k+1} \in \N$ such that $r_{k+1} > r_k + 1$ holds and 
\begin{equation}
\label{eq:randomeq464646}
\left |
\sum_{m=1}^{q_k} \beta_m \right |
< 
\sum_{m=r_k + 1}^{r_{k+1}} \frac{1}{m}.
\end{equation}
\STATE\label{algo:step:7} Choose $z_{k+1} \in (z_k , 1)$ such that 
\begin{equation}
\label{eq:randomeq37eg378}
\left |
\sum_{m=1}^{q_k} \beta_m z_{k+1}^{\alpha_m}
\right |
< 
\sum_{m=r_k + 1}^{r_{k+1}} \frac{1}{m}z_{k+1}^{\alpha_{m + p_{k+1} - r_k - 1}}.
\end{equation}
\STATE\label{algo:step:8} Set $q_{k+1} := p_{k+1} + r_{k+1} - r_k - 1$. 
\STATE\label{algo:step:9} Define 
\begin{equation}
\label{eq:randomeq237dwws378}
\beta_{i + p_{k+1} - 1}
:=
\begin{cases}
\displaystyle
\frac{1}{r_k + i} & \text{ for } i=1,\ldots , r_{k+1} - r_k \text{ if } k+1 \text{ is odd},
\\[0.35cm]
\displaystyle
\frac{- 1}{r_k + i} & \text{ for } i=1,\ldots , r_{k+1} - r_k \text{ if } k+1 \text{ is even}.
 \end{cases}
\end{equation}
    \ENDFOR
\end{algorithmic}
\end{Algorithm}

Before we demonstrate that the above procedure
indeed provides a solution for problem \eqref{eq:Q},
we prove that \cref{algo:AlgoHarmonic} is executable and analyze the behavior of its iterates. 

\begin{lemma}[{Executability of \cref{algo:AlgoHarmonic} and properties of iterates}]
\label{lem:wellAlgo}
Let  $\{\alpha_m\}_{m \in \mathbb{N}} \subset \mathbb{N}$ be a sequence satisfying $\alpha_m \geq m$ for all $m \in \mathbb{N}$ and let $z_1 \in (0,1)$ be given.
Then all steps in \cref{algo:AlgoHarmonic} can be executed and the produced sequences 
\begin{gather*}
\{\beta_m\}_{m \in \mathbb{N}} \subset \mathbb{R},\qquad
\{z_k\}_{k \in \mathbb{N}} \subset \mathbb{R}, 
\qquad 
\text{and}
\qquad
\{p_k\}_{k \in \mathbb{N}},
\{q_k\}_{k \in \mathbb{N}},
\{r_k\}_{k \in \mathbb{N}} \subset \mathbb{N}
\end{gather*}
satisfy the following conditions for all $k \in \N$: 
   \begin{enumerate}[label=\roman*)]
\item\label{lem:wellAlgo:item:1} 
$z_k \in (0,1)$;
\item\label{lem:wellAlgo:item:2} 
 $z_{k} > z_{k-1}$ if  $k \geq 2$;
\item\label{lem:wellAlgo:item:3} 
$p_{k} > q_{k-1}$ if  $k \geq 2$; 
\item\label{lem:wellAlgo:item:4} 
 $q_{k} > p_k$ if $k \geq 2$;
\item\label{lem:wellAlgo:item:5} 
 $r_{k} > r_{k-1}$ if  $k \geq 2$;
\item\label{lem:wellAlgo:item:6} 
 $\beta_m \in [-1,1]$ for all $m=1,\ldots, q_k$;
\item\label{lem:wellAlgo:item:7} 
 $\beta_m = 0$ for all $q_{k-1} < m < p_{k}$ if $k \geq 2$;
\item\label{lem:wellAlgo:item:8} 
\[
\sum_{m=1}^{q_k} \left | \beta_m \right |^\gamma = \sum_{m=1}^{r_k} \left (  \frac1m \right )^\gamma 
\qquad 
\forall \gamma \in [1,\infty);
\]
\item\label{lem:wellAlgo:item:9} 
\[
\sum_{m=1}^{q_k} \beta_m  z_k^{\alpha_m} 
\begin{cases}
> 0 & \text{ if } k \text{ is odd,}
\\
< 0 & \text{ if } k \text{ is even;}
\end{cases}
\]
\item\label{lem:wellAlgo:item:10} 
\[
\left |
\sum_{m=1}^{q_{k-1}} \beta_m  z_{k-1}^{\alpha_m} 
\right |
> 
 \sum_{m=p_k}^{\infty} z_{k-1}^{\alpha_m} 
\qquad 
\text{ if } k \geq 2.
\]
\end{enumerate}
\end{lemma}

\begin{proof}
We prove by induction w.r.t.\ 
$k \in \N$ that the steps in \cref{algo:AlgoHarmonic}
can be executed and that the produced sequences 
satisfy \ref{lem:wellAlgo:item:1}-\ref{lem:wellAlgo:item:10}.
\smallskip

\emph{Base step $k=1$:} 
For $k=1$, $z_1 \in (0,1)$ is an input datum
of  \cref{algo:AlgoHarmonic}
and we obtain from 
\cref{algo:step:2} of the algorithm
that $p_1 = q_1 = r_1 = \beta_1 = 1$ holds. 
Note that this 
defines the values $\beta_m$ for all $m=1,\ldots, q_1$.
That the conditions 
$\beta_m \in \R$  for all $m=1,\ldots, q_1$,
 $z_1 \in \R$, 
and $p_1, q_1, r_1 \in \N$ are satisfied here is obvious. 
We further immediately see that 
\ref{lem:wellAlgo:item:1},
\ref{lem:wellAlgo:item:6},
\ref{lem:wellAlgo:item:8}, and 
\ref{lem:wellAlgo:item:9}
hold for these numbers and $k=1$. 
As the remaining conditions 
\ref{lem:wellAlgo:item:2},
\ref{lem:wellAlgo:item:3},
\ref{lem:wellAlgo:item:4},
\ref{lem:wellAlgo:item:5},
\ref{lem:wellAlgo:item:7},
and 
\ref{lem:wellAlgo:item:10}
are only for $k\geq 2$, this concludes the base step of the induction proof. 
\smallskip

\emph{Induction step $k \mapsto k+1$:}
Let $k \in \N$ be arbitrary, 
suppose that \cref{algo:AlgoHarmonic} has generated
numbers $z_l \in \R$, $p_l, q_l, r_l \in \N$ for $l=1,\ldots, k$
and $\beta_m \in \R$
for $m=1,\ldots, q_k$ that satisfy 
\ref{lem:wellAlgo:item:1}-\ref{lem:wellAlgo:item:10},
and assume that we now enter the for-loop in \cref{algo:step:3}
of the algorithm with counter $k$. Then
it follows from our assumption $\alpha_m \geq m$ for all $m \in \N$,
the validity of \ref{lem:wellAlgo:item:1} for $k$, and the convergence of the 
geometric series that 
\[
	0 <
	\sum_{m=M}^\infty z_k^{\alpha_m} 
	\leq
	\sum_{m=M}^\infty z_k^{m} \to 0
\]
holds for $M \to \infty$. Since \ref{lem:wellAlgo:item:9} for $k$ implies 
\[
\left | \sum_{m=1}^{q_k} \beta_m z_k^{\alpha_m}  \right | > 0,
\]
this shows that we can find a number $p_{k+1} \in \N$
which satisfies $p_{k+1} > q_k$  and \eqref{eq:randomeq363738}.
In particular, \cref{algo:step:4} of 
\cref{algo:AlgoHarmonic}  can be executed for the index $k$. 
Note that the above choice of $p_{k+1}$ also ensures
\ref{lem:wellAlgo:item:3}
and 
\ref{lem:wellAlgo:item:10}
for $k+1$. 

In \cref{algo:step:5}, \cref{algo:AlgoHarmonic} 
next sets $\beta_m := 0$ for all $m= q_k + 1,\ldots , p_{k+1} - 1$.  This 
is trivially possible, determines $\beta_m$ up to the index $m= p_{k+1} - 1$,
and ensures that condition \ref{lem:wellAlgo:item:7} holds for $k+1$.

Choosing a number $r_{k+1} \in \N$ 
in the subsequent \cref{algo:step:6} of \cref{algo:AlgoHarmonic} 
such that $r_{k+1} > r_k + 1$ is valid and \eqref{eq:randomeq464646} 
is satisfied works since the divergence of the harmonic series implies 
\begin{equation*}
\sum_{m=r_k + 1}^{M} \frac{1}{m} \to \infty \qquad \text{ for } M \to \infty.
\end{equation*}
The choice of $r_{k+1}$ also ensures \ref{lem:wellAlgo:item:5} for $k+1$. 

Since polynomials are continuous, we clearly have 
\[
\lim_{(0,1) \ni z \to 1}
\sum_{m=1}^{q_k} \beta_m  z ^{\alpha_m}
=
\sum_{m=1}^{q_k} \beta_m 
,\quad
\lim_{(0,1) \ni z \to 1}
\sum_{m=r_k + 1}^{r_{k+1}} \frac{1}{m} z^{\alpha_{m + p_{k+1} - r_k - 1}}
=
\sum_{m=r_k + 1}^{r_{k+1}} \frac{1}{m}.
\]%
In view of \eqref{eq:randomeq464646} 
and the validity of \ref{lem:wellAlgo:item:1} for $k$, this entails that we can 
find a number $z_{k+1} \in (z_k , 1)$ that satisfies
\eqref{eq:randomeq37eg378} in \cref{algo:step:7} of \cref{algo:AlgoHarmonic}.
That \ref{lem:wellAlgo:item:1} and \ref{lem:wellAlgo:item:2} hold for $k+1$
follows immediately from this definition of $z_{k+1}$.

Defining $q_{k+1} := p_{k+1} + r_{k+1} - r_k - 1$ in \cref{algo:step:8}
of \cref{algo:AlgoHarmonic} works without problems and, due to 
the inequality $r_{k+1} > r_k + 1$, ensures \ref{lem:wellAlgo:item:4}  for $k+1$. 

The definition of $\beta_{i + p_{k+1} - 1}$ for $i=1,\ldots, r_{k+1} - r_k$
via \eqref{eq:randomeq237dwws378} in \cref{algo:step:9} is also unproblematic. 
Note that, due to the definition of $q_{k+1}$, \cref{algo:step:5} of the algorithm,
and the induction hypothesis, 
this now determines $\beta_m$ for all $m=1,\ldots, q_{k+1}$.
Further, we trivially have $\beta_m \in [-1,1]$ for all these $m$,
 so that \ref{lem:wellAlgo:item:6}  holds for $k+1$.

In total, we have now proven that the steps 
in \cref{algo:AlgoHarmonic} can be  executed for the index $k$ 
and produce 
new values $p_{k+1}, q_{k+1}, r_{k+1} \in \N$, $z_{k+1} \in \R$,
and $\beta_m \in \R$, $m=q_{k}+1,\ldots, q_{k+1}$,
such that \ref{lem:wellAlgo:item:1}-\ref{lem:wellAlgo:item:7}
and \ref{lem:wellAlgo:item:10} hold for $k+1$.
It remains to establish the validity of
\ref{lem:wellAlgo:item:8} and \ref{lem:wellAlgo:item:9} 
for the newly created numbers. 

To obtain \ref{lem:wellAlgo:item:8} for $k+1$, we note that 
the definitions of the new iterates
and the induction hypothesis yield that, 
for all $\gamma \in [1, \infty)$, we have 
\begin{equation*}
\begin{aligned}
\sum_{m=1}^{q_{k+1}} \left | \beta_m \right |^\gamma
&=
\sum_{m=1}^{q_{k}} \left | \beta_m \right |^\gamma
+
\sum_{m=p_{k+1}}^{q_{k+1}} \left | \beta_m \right |^\gamma
&&\text{(by the definitions in \cref{algo:step:5})}
\\
&=
\sum_{m=1}^{r_k} \left (  \frac1m \right )^\gamma 
+
\sum_{m=p_{k+1}}^{ p_{k+1} + r_{k+1} - r_k - 1} \left | \beta_m \right |^\gamma
&&\text{(by \ref{lem:wellAlgo:item:8} for $k$ and \cref{algo:step:8})}
\\
 &=
\sum_{m=1}^{r_k} \left (  \frac1m \right )^\gamma 
+
\sum_{i=1}^{r_{k+1} - r_k} \left | \beta_{i + p_{k+1} - 1} \right |^\gamma
&&\text{(by shifting the index)}
\\
 &=
\sum_{m=1}^{r_k} \left (  \frac1m \right )^\gamma 
+
\sum_{i=1}^{r_{k+1} - r_k} \left | \frac{1}{r_k + i}  \right |^\gamma
&&\text{(by definition \eqref{eq:randomeq237dwws378})}
\\
 &=
\sum_{m=1}^{r_{k+1}} \left (  \frac1m \right )^\gamma
&&\text{(by shifting the index)}.
\end{aligned}
\end{equation*}
This shows that \ref{lem:wellAlgo:item:8} holds for $k+1$. 

If $k+1$ is an odd number, then we can further use the definitions in 
\cref{algo:AlgoHarmonic} to compute that 
\begin{equation*}
\begin{aligned}
&\sum_{m=1}^{q_{k+1}} \beta_m z_{k+1}^{\alpha_m} 
\\
&=
\sum_{m=1}^{q_{k}} \beta_m z_{k+1}^{\alpha_m} 
+
\sum_{m=p_{k+1}}^{q_{k+1}} \beta_m z_{k+1}^{\alpha_m} 
&&\text{(by \cref{algo:step:5})}
\\
&=
\sum_{m=1}^{q_{k}} \beta_m z_{k+1}^{\alpha_m} 
+
\sum_{m=p_{k+1}}^{p_{k+1} + r_{k+1} - r_k - 1} \beta_m z_{k+1}^{\alpha_m} 
&&\text{(by \cref{algo:step:8})}
\\
&=
\sum_{m=1}^{q_{k}} \beta_m z_{k+1}^{\alpha_m} 
+
\sum_{m=r_k + 1}^{r_{k+1}} \beta_{m + p_{k+1} - r_k - 1} z_{k+1}^{\alpha_{m + p_{k+1} - r_k - 1}} 
&&\text{(by shifting the index)}
\\
&=\sum_{m=1}^{q_{k}} \beta_m z_{k+1}^{\alpha_m} 
+
\sum_{m=r_k + 1}^{r_{k+1}} \frac{1}{m} z_{k+1}^{\alpha_{m + p_{k+1} - r_k - 1}} 
&&\text{(by \eqref{eq:randomeq237dwws378})}
\\
&\geq
\sum_{m=r_k + 1}^{r_{k+1}} \frac{1}{m} z_{k+1}^{\alpha_{m + p_{k+1} - r_k - 1}} 
-
\left |
\sum_{m=1}^{q_k} \beta_m  z_{k+1}^{\alpha_m}
\right |
> 0
&&\text{(by \eqref{eq:randomeq37eg378})}.
\end{aligned}
\end{equation*}
Completely analogously, we also obtain for an even $k+1$ that 
\begin{equation*}
\begin{aligned}
\sum_{m=1}^{q_{k+1}} \beta_m  z_{k+1}^{\alpha_m} 
&=\sum_{m=1}^{q_{k}} \beta_m z_{k+1}^{\alpha_m} 
-
\sum_{m=r_k + 1}^{r_{k+1}} \frac{1}{m} z_{k+1}^{\alpha_{m + p_{k+1} - r_k - 1}} 
\\
&\leq
\left |
\sum_{m=1}^{q_k} \beta_m  z_{k+1}^{\alpha_m}
\right |
-
\sum_{m=r_k + 1}^{r_{k+1}} \frac{1}{m} z_{k+1}^{\alpha_{m + p_{k+1} - r_k - 1}} 
 < 0.
\end{aligned}
\end{equation*}
This establishes the validity of \ref{lem:wellAlgo:item:9} 
for $k+1$ and completes the induction proof.
\end{proof}

With \cref{lem:wellAlgo} at hand, we can answer question \eqref{eq:Q}.
\smallskip

\begin{theorem}[{Solution of \eqref{eq:Q}}]
\label{th:power_series}
Let
$\{\alpha_m\}_{m \in \mathbb{N}} \subset \mathbb{N}$ be a sequence 
that satisfies $\alpha_m \geq m$ for all $m \in \mathbb{N}$ and let $z_1 \in (0,1)$ be given. 
Let  $\{\beta_m\}_{m \in \N} \subset \R$, $\{z _k\}_{k \in \N} \subset \R$,
and $\{q_k\}_{k \in \mathbb{N}} \subset \N$
be sequences that have been generated by \cref{algo:AlgoHarmonic}
with input data $\{\alpha_m\}_{m \in \mathbb{N}}$ and $z_1$. Define 
\begin{equation}
\label{eq:PdefPl}
P(z) := \sum_{m=1}^\infty \beta_m z^{\alpha_m}
\qquad
\text{and}
\qquad
P_L(z)
:= 
 \sum_{m=1}^{q_L} \beta_m z^{\alpha_m}
~~
\text{for}
~~
L \in \N.
\end{equation}
Then the following statements are true:
   \begin{enumerate}[label=\roman*)]
\item\label{th:power_series:i} 
It holds
$0 <   z_k <   z_{k+1} < 1$ for all $k \in \N$ and $  z_k \to 1 $ for $k \to \infty$.
\item\label{th:power_series:ii} 
It holds $\beta_m \in [-1,1]$ for all $m \in \N$ and 
\begin{equation*}
	\sum_{m=1}^{\infty} 
	\left | \beta_m \right |^\gamma 
	= 
	\sum_{m=1}^{\infty} \left (  \frac1m \right )^\gamma < \infty
	\qquad 
	\forall \gamma \in (1,\infty).
\end{equation*}
\item\label{th:power_series:iii} The power series $P$
has  radius of convergence one and it holds $P(z_k) > 0$ for all $k \in \N$ that are odd
and $P(z_k) < 0$ for all $k \in \N$ that are even.
\item\label{th:power_series:iv} For all  $L \in \N$, it holds $P_L(z_k) > 0$ for all $k=1,\ldots, L$ that are odd
and $P_L(z_k) < 0$ for all $k=1,\ldots, L$ that are even.
\end{enumerate}
\end{theorem}

\begin{proof}
Let $\{\alpha_m\}_{m \in \mathbb{N}}$
and $z_1$ be as in the statement of the theorem,
and let 
$\{\beta_m\}_{m \in \mathbb{N}} $,
$\{z_k\}_{k \in \mathbb{N}} $,
$\{p_k\}_{k \in \mathbb{N}} $,
and
$\{q_k\}_{k \in \mathbb{N}} $
be sequences 
that have been generated for $\{\alpha_m\}_{m \in \mathbb{N}}$
and $z_1$ by means of
\cref{algo:AlgoHarmonic}.
Then 
it follows from \cref{lem:wellAlgo}
that 
$\{\beta_m\}_{m \in \mathbb{N}}$ 
has the properties in \ref{th:power_series:ii}
and that 
$\{z_k\}_{k \in \mathbb{N}} $
satisfies 
$0 <   z_k <   z_{k+1} < 1$ for all $k \in \N$.
From the inclusion 
$\beta_m \in [-1,1]$ for all $m \in \N$,
the inequality $\alpha_m \geq m$,
and the properties of the geometric series, 
we further obtain that $P(z)$ has radius of convergence at least one. 
In particular, $P$ is holomorphic in 
$\{z \in \C \mid |z| < 1\}$ and smooth on $(-1,1)$. 
Suppose now that $k \in \N$ is odd. 
Then \cref{lem:wellAlgo} implies
\begin{equation}
\label{eq:randomeq:P282hw83}
\begin{aligned}
P(z_k) = 
\sum_{m=1}^\infty \beta_m z_k^{\alpha_m}
&=
\sum_{m=1}^{q_k} \beta_m z_k^{\alpha_m}
+
\sum_{m=p_{k+1}}^{\infty} \beta_m z_k^{\alpha_m}
&&
\text{(by \cref{lem:wellAlgo}\ref{lem:wellAlgo:item:7})}
\\
 &\geq 
\sum_{m=1}^{q_k} \beta_m z_k^{\alpha_m}
-
\sum_{m=p_{k+1}}^{\infty} z_k^{\alpha_m}
&&\text{(by \cref{lem:wellAlgo}\ref{lem:wellAlgo:item:1}\ref{lem:wellAlgo:item:6})}
\\
 &=
\left |\sum_{m=1}^{q_k} \beta_m z_k^{\alpha_m}\right |
-
\sum_{m=p_{k+1}}^{\infty} z_k^{\alpha_m}
&&\text{(by \cref{lem:wellAlgo}\ref{lem:wellAlgo:item:9})}
\\
& >0
&&\text{(by \cref{lem:wellAlgo}\ref{lem:wellAlgo:item:10})}.
\end{aligned}
\end{equation}
Completely analogously, we also obtain for all even $k \in \N$ that
\begin{equation}
\label{eq:randomeq:P282hw83-2}
\begin{aligned}
P(z_k) = 
\sum_{m=1}^\infty \beta_m z_k^{\alpha_m}
&=
\sum_{m=1}^{q_k} \beta_m z_k^{\alpha_m}
+
\sum_{m=p_{k+1}}^{\infty} \beta_m z_k^{\alpha_m}
&&\text{(by \cref{lem:wellAlgo}\ref{lem:wellAlgo:item:7})}
\\
 &\leq 
\sum_{m=1}^{q_k} \beta_m z_k^{\alpha_m}
+
\sum_{m=p_{k+1}}^{\infty} z_k^{\alpha_m}
&&\text{(by \cref{lem:wellAlgo}\ref{lem:wellAlgo:item:1}\ref{lem:wellAlgo:item:6})}
\\
 &=
 - \left |\sum_{m=1}^{q_k} \beta_m z_k^{\alpha_m}\right |
+
\sum_{m=p_{k+1}}^{\infty} z_k^{\alpha_m}
&&\text{(by \cref{lem:wellAlgo}\ref{lem:wellAlgo:item:9})}
\\
& <0
&&\text{(by \cref{lem:wellAlgo}\ref{lem:wellAlgo:item:10})}.
\end{aligned}
\end{equation}
Due to the continuity of $P$ on $(-1,1)$ and 
$0 < z_k < z_{k+1} < 1$ for all $k \in \N$,
the above change of sign implies that 
$P$ has at least one zero $\bar z_k$ in every interval $(z_k, z_{k+1})$.
Due to the properties of $\{z_k\}_{k \in \mathbb{N}} $, 
this sequence of zeros satisfies 
$0 < \bar z_k < \bar z_{k+1} < 1$ for all $k \in \N$. 
Further, we necessarily have $\bar z_k \to 1$.
Indeed, as 
$\{\bar z_k\}_{k \in \mathbb{N}} $ is nondecreasing and bounded by one,
we know that there exists $a \in (0,1]$
such that $\bar z_k \to a$ holds for $k \to \infty$.
If $a < 1$ held, then the fact that $P$
is analytic in $\{z \in \C \mid |z| < 1\}$ 
and the identity theorem 
\cite[Result 1.2E.1e)]{Borwein1995}
would imply $P \equiv 0$ in $(-1,1)$,
in contradiction to the inequalities $P(z_k) \neq 0$ for all $k$ derived above.
Thus, $a = 1$ and $\bar z_k \to 1$ follows. 
Note that the same contradiction argument also 
shows that the radius of convergence of $P$ cannot be larger than one. 
As $\bar z_k < z_{k+1} < 1$, it follows that $z_k \to 1$ holds as well. 
This proves the statements 
\ref{th:power_series:i}, \ref{th:power_series:ii}, and \ref{th:power_series:iii} 
of the theorem. 
To obtain \ref{th:power_series:iv},
one can use the exact same arguments as in 
\eqref{eq:randomeq:P282hw83}
and
\eqref{eq:randomeq:P282hw83-2}.
\end{proof}

Note that, due to the use of the harmonic series, we do not get $\{\beta_m\}_{m \in \N} \in \ell^1$
in \cref{th:power_series} and that point \ref{th:power_series:iv} of \cref{th:power_series} 
is very important
for practical applications
as it implies that 
the partial sums $P_L$ accurately reproduce the oscillatory behavior of 
the power series $P$. 
We conclude this section with a numerical experiment
which demonstrates that \cref{algo:AlgoHarmonic} can indeed be implemented
and illustrates how the 
polynomials 
$P_L$ in \cref{th:power_series} behave in practice. 
As a model problem, we consider the situation 
$\alpha_m = m^2$, $m \in \N$,
and $z_1 = 0.5$.

We begin with some remarks on how the steps 
in  \cref{algo:AlgoHarmonic} can be realized 
for this example. 

\begin{enumerate}[label=\roman*)]
\item
To effectively implement
\cref{algo:step:4} 
of \cref{algo:AlgoHarmonic} 
for the sequence $\alpha_m = m^2$, we note that 
\cref{lem:wellAlgo}\ref{lem:wellAlgo:item:1}  and the formula for the geometric series imply
\begin{equation*}
\begin{aligned}
0 &<
 \sum_{m=p_{k+1}}^{\infty}  z_k^{\alpha_m} 
= 
 z_k^{p_{k+1}^2}
  \sum_{m=p_{k+1}}^{\infty}  z_k^{m^2 - p_{k+1}^2} 
  \\
  &\leq
  z_k^{p_{k+1}^2}
  \sum_{m=p_{k+1}}^{\infty} z_k^{m- p_{k+1}} 
   = 
  z_k^{p_{k+1}^2}
  \sum_{m=0}^{\infty} z_k^{m} 
   =
  \frac{\mathrm{e}^{\ln(z_k)p_{k+1}^2}}{1 - z_k}.
 \end{aligned}
\end{equation*}
The above shows that, to guarantee 
$p_{k+1} > q_k$ and \eqref{eq:randomeq363738},
it suffices to define
\begin{equation*}
\qquad
   p_{k+1}
  :=
  \max \left (
  q_k + 1, 
  \left \lfloor
  \sqrt{
  \frac{1}{\ln(z_k)}
  \ln \left [ (1 - z_k)
  \left |
  \sum_{m=1}^{q_k} \beta_m z_k^{\alpha_m} 
  \right |
  \right ]}
  \right \rfloor + 1
  \right ),
\end{equation*}
where $\lfloor \cdot \rfloor$ denotes the operation of rounding down. 

\item For the realization of \cref{algo:step:6}, we recall that harmonic sums can be interpreted as 
Riemann sums for the integration of $f(x) := 1/(x+1)$. This yields 
\begin{equation*}
\begin{aligned}
\sum_{m=r_k + 1}^{r_{k+1}} \frac{1}{m}
&\geq 
\int_{r_k}^{r_{k+1}} \frac{1}{x+1} \dd x
=
\ln(r_{k+1} +1)
-
\ln(r_k +1)
\end{aligned}
\end{equation*}
and allows us to ensure $r_{k+1} > r_k + 1$  and \eqref{eq:randomeq464646} by setting
\begin{equation*}
	r_{k+1}
	:=
	\max
	\left (
	r_k + 2,
	\left \lfloor
	\exp \left ( 
	\left |
	\sum_{m=1}^{q_k} \beta_m 
	\right |
	+
	\ln(r_k+1)
	\right )
	\right \rfloor
	\right ).
\end{equation*}

\item To finally identify a number $z_{k+1} \in (z_k, 1)$
satisfying the condition \eqref{eq:randomeq37eg378} 
in \cref{algo:step:7} of \cref{algo:AlgoHarmonic},
one can use simple interval bisection.  
\end{enumerate}
\smallskip

The sequences 
$\{z_k\}_{k \in \mathbb{N}}$, 
$\{p_k\}_{k \in \mathbb{N}}$,
$\{q_k\}_{k \in \mathbb{N}}$,
and 
$\{r_k\}_{k \in \mathbb{N}}$
and 
polynomials $P_L$
that are generated by 
\cref{algo:AlgoHarmonic} 
for $z_1 = 0.5$  and 
$\alpha_m = m^2$, $m \in \N$,
when the above selection rules are used
can be seen in \cref{tab:1} and \Cref{fig:1} below. 
The obtained results suggest that the 
construction in \cref{algo:AlgoHarmonic} causes 
$\{z_k\}_{k \in \mathbb{N}}$, 
$\{p_k\}_{k \in \mathbb{N}}$,
and
$\{q_k\}_{k \in \mathbb{N}}$
to converge very rapidly to one and infinity, respectively.
In the last line of \cref{tab:1}, one
can further see that only very few of the coefficients $\beta_m$
in the partial sums $P_L$ are nonzero in our experiment.
(Note that the number of nonzero coefficients in $P_L$ is precisely $r_L$ by  \eqref{eq:bmstructure}.) 
Because of this effect, the functions $P_L$ can indeed be worked with, 
despite the very large summation bounds $q_L$ in \eqref{eq:PdefPl}.\medskip

\begin{table}[ht]
\centering
\begin{tabular}{c | c | c | c | c | c | c  }
\hline\noalign{\smallskip}
$k$ & $1$  & $2$  & $3$  &  $4$   & $5$  & $6$ \\
\noalign{\smallskip}\hline\noalign{\smallskip}
$1 - z_k$  & $ 5\cdot 10^{-1} $  & $1.56\cdot 10^{-2}$  & $1.22\cdot 10^{-4}$  &  $1.52\cdot 10^{-5}$   & $2.38\cdot 10^{-7}$  & $7.45\cdot 10^{-9}$  \\[0.1cm]
$p_k$  & $1$  & $2$  & $21$  &  $333$   & $994$  & $9069$  \\[0.1cm]
$q_k$  & $1$  & $5$  & $22$  &  $334$   & $996$  & $9070$     \\[0.1cm]
$r_k$ & $1$  & $5$  & $7$  &  $9$   & $12$  & $14$  \\[0.1cm]
\noalign{\smallskip}\hline
\end{tabular}
\bigskip
\caption{The sequences  
$\{z_k\}_{k \in \mathbb{N}}$, 
$\{p_k\}_{k \in \mathbb{N}}$,
$\{q_k\}_{k \in \mathbb{N}}$,
and 
$\{r_k\}_{k \in \mathbb{N}}$
generated by \cref{algo:AlgoHarmonic} 
for $\alpha_m = m^2$ and $z_1 = 0.5$.
}
\label{tab:1}
\end{table}

\begin{figure}[htp]
\begin{center}
  \includegraphics[width=11cm]{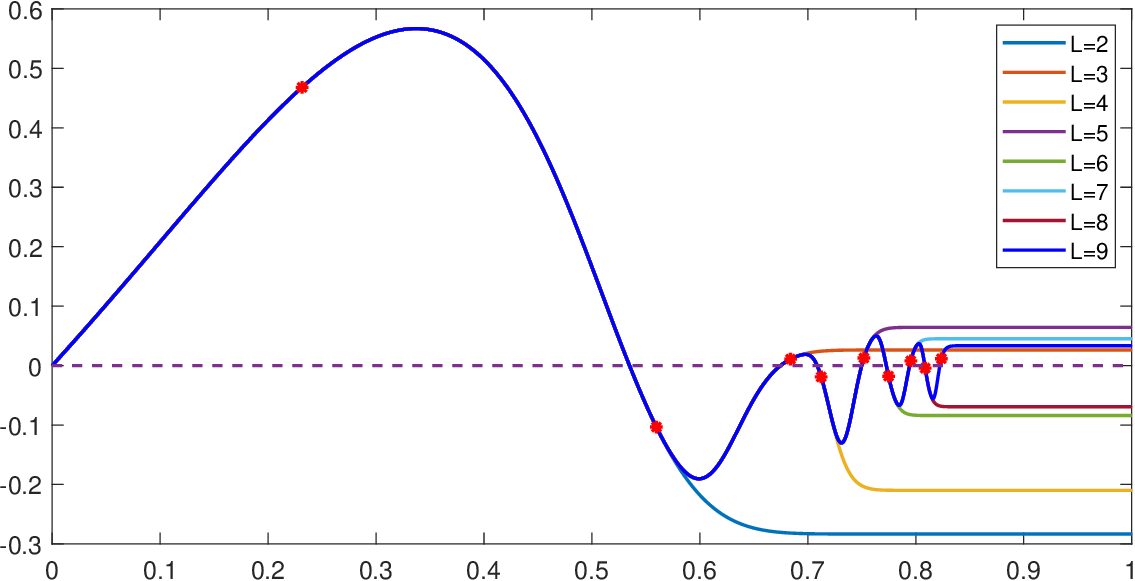}\\
  \caption{The polynomials $P_L$
  generated by \cref{algo:AlgoHarmonic} for 
  $\alpha_m = m^2$ and $z_1 = 0.5$ with $L=2,\ldots, 9$. To properly visualize the behavior near one, 
  the plot shows the graphs of the rescaled functions
  $[0,1] \ni x \mapsto P_L(1 - \exp(1 - (1-x)^{-2})) \in \R$.
  The red dots mark the (rescaled) positions of the points 
  $z_k$, $k=1,\ldots, 9$. The sign-changing behavior from 
  \cref{th:power_series}\ref{th:power_series:iv} is clearly visible. 
  }\label{fig:1}
  \end{center}
\end{figure}

\section{Construction of a bang-bang solution with infinitely many switches}
\label{sec:4}

With
\cref{th:power_series}
at hand, we can turn our attention to the construction of 
an instance of problem \eqref{eq:P} that gives rise to a chattering optimal control.
We begin with some preliminary results.\smallskip

\begin{proposition}[Solvability and first-order optimality conditions]
\label{prop:necsuffoptcond}
Let $T>0$ and $y_d \in L^2(0,1)$ be given. 
Then problem \eqref{eq:P} possesses an optimal control-state pair $(\bar u, \bar y)$.
Further, a tuple $(\bar u, \bar y)$ solves \eqref{eq:P} if and only if there
exists an adjoint state $\bar p$ such that the following system is satisfied:\smallskip
\begin{equation}
\label{eq:optsysP}
\begin{gathered}
\bar y, \bar p \in  L^2(0,T;H^1(0,1)) \cap H^1(0,T;H^1(0,1)^*),
\qquad \bar u \in L^2(0,T), \\[0.05cm]
\begin{aligned}
		\partial_t \bar y - \partial_x^2 \bar y &= 0 &&\hspace{-0.15cm}\text{in } (0,T) \times (0,1), \\
		\partial_x \bar y &= 0 &&\hspace{-0.15cm}\text{on } (0,T) \times \{0\}, \\
		\partial_x \bar y &= \bar u &&\hspace{-0.15cm}\text{on } (0,T) \times \{1\}, \\
		\bar y &= 0  &&\hspace{-0.15cm}\text{on } \{0\} \times (0,1), \\
\end{aligned}
\hspace{1cm}
\begin{aligned}
		-\partial_t \bar p - \partial_x^2 \bar p  &= 0 &&\hspace{-0.15cm}\text{in } (0,T) \times (0,1), \\
		\partial_x \bar p  &= 0 &&\hspace{-0.15cm}\text{on } (0,T) \times \{0 \}, \\
		\partial_x \bar p  &= 0 &&\hspace{-0.15cm}\text{on } (0,T) \times \{1\}, \\
		\bar p &= \bar y - y_d   &&\hspace{-0.15cm}\text{on } \{T\} \times (0,1),
\end{aligned}
\\[0.05cm]
\bar u \in \Uad,\quad 
\left ( \bar p(\cdot, 1), u - \bar u \right )_{L^2(0,T)}
\geq 0
\quad \forall u \in \Uad. \\[0.05cm]
\end{gathered}
\end{equation}
\end{proposition}

Here and in what follows, the appearing parabolic 
boundary value problems are, of course, understood in the weak sense; cf.\ 
\cite[Section 7.1.1b]{Evans2010}.

\begin{proof}[Proof of \cref{prop:necsuffoptcond}]
The solvability of \eqref{eq:P} is obtained along standard lines with the direct method 
of the calculus of variations, and that every optimal control-state pair $(\bar u, \bar y)$
of \eqref{eq:P}
satisfies \eqref{eq:optsysP} for some $\bar p$ is a straightforward 
consequence of adjoint calculus. That \eqref{eq:optsysP} is 
 sufficient for optimality 
follows from the convexity of \eqref{eq:P}. 
For more details, we refer the reader to 
\cite[Section 3.2.3]{Troeltzsch2010}. 
\end{proof}

As a consequence of \cref{prop:necsuffoptcond}, we obtain the 
already mentioned bang-bang principle.\smallskip

\begin{proposition}[Bang-bang principle]
\label{prop:bang_bang_principle}
Let $T>0$ and $y_d \in L^2(0,1)$ be given. 
Suppose that the optimal objective function value of \eqref{eq:P} is 
positive, i.e., that there is no state $y$ that is admissible for \eqref{eq:P}
and satisfies $y(T, \cdot) = y_d$ a.e.\ in $(0,1)$. Then 
\eqref{eq:P} admits precisely one optimal control $\bar u$,
the adjoint state $\bar p$ associated with this optimal control possesses 
a trace $\bar p(\cdot, 1)\colon [0,T) \to \R$ that is analytic and not identical zero on $[0,T)$, 
and it holds $\bar u = - \sgn(\bar p(\cdot, 1))$ a.e.\ in $(0,T)$. 
In particular, $\bar u$ possesses a (necessarily unique) right-continuous 
representative $\bar u \colon [0,T) \to \R$
and this representative satisfies $ \bar u(t) \in \{-1,1\}$ for all $t \in [0,T)$
and changes its function value at at most countably many points in 
$[0, T)$ which can accumulate at the terminal time T only.
\end{proposition}

\begin{proof}
The proof of the proposition is based on the idea to express the solution $\bar p$
of the adjoint equation in \eqref{eq:optsysP} in terms of 
Green's function for the heat equation with homogeneous Neumann boundary conditions 
and to then establish the analyticity of $\bar p(\cdot, 1)$ on $[0,T)$ by means of the 
resulting convolution formula; cf.\ the derivation of \cref{lemma:oscadjoint} below. 
As the 
assumption that the  optimal objective function value of \eqref{eq:P}
is positive implies that $\bar p(\cdot, 1)$ cannot be identical zero, 
the identity theorem then yields that $\bar p(\cdot, 1)$
can have at most countably many zeros on $[0, T)$ which can 
accumulate at $T$ only. With this information at hand,  it follows 
from the variational 
inequality in \eqref{eq:optsysP} that $\bar u = - \sgn(\bar p(\cdot, 1))$ holds
and the assertions on the mapping behavior and switching structure 
of $\bar u$ are immediately obtained. The uniqueness of $\bar u$ finally 
follows from a simple contradiction argument. For 
a detailed proof of the proposition, we refer the reader to
\cite[Section 3.2.4]{Troeltzsch2010}.
\end{proof}

Note that, if the optimal objective function value of \eqref{eq:P} is zero, then
it follows from \eqref{eq:optsysP} that the adjoint state $\bar p$ vanishes identically 
for all optimal control-state pairs $(\bar u, \bar y)$. 
In this case, the variational inequality in \eqref{eq:optsysP} 
does not contain any information and 
an optimal control $\bar u$ may very well take values in $(-1,1)$.
As a matter of fact, one easily checks that, for every arbitrary but fixed
$\bar u \in \Uad$, there exists $y_d \in L^2(0,1)$ such that 
problem \eqref{eq:P} with 
desired state $y_d$
is solved by $\bar u$ with optimal objective function value zero. 
This means in particular that
it is trivial to construct an instance of \eqref{eq:P}
which has an optimal control that switches infinitely often between 
$-1$ and $1$  if one allows that the desired state $y_d$ can be fitted precisely. 
Constructing such an optimal control is only delicate if one restricts the attention to situations
to which the bang-bang principle in \cref{prop:bang_bang_principle} applies, i.e., 
to cases
for which the optimal objective function value is positive.

To establish that infinitely many switching points are indeed possible 
under the assumptions of \cref{prop:bang_bang_principle}, we use \cref{th:power_series}
to prove the following lemma on the heat equation.
\smallskip

\begin{lemma}[Oscillations for the adjoint equation]
\label{lemma:oscadjoint}
Let $T>0$ and $z_1 \in (0,1)$ be given.
Let 
$\{\beta_m\}_{m \in \N} \subset \R$ and
 $\{z _k\}_{k \in \N} \subset \R$
be sequences that have been generated by \cref{algo:AlgoHarmonic}
with input data $\alpha_m := m^2$, $m \in \N$, and $z_1$. 
Define 
\begin{equation}
\label{eq:w-def}
w(x) := \sum_{m=1}^\infty (-1)^m\beta_m \cos(m \pi x)\qquad \text{f.a.a.\ } x \in (0,1)
\end{equation}
and
\begin{equation}
\label{eq:tk-def}
t_k := \pi^{-2} \ln( z_k ) + T,\qquad k \in \N.
\end{equation}
Then $w$ is in
$L^2(0,1)\setminus \{0\}$,
it holds 
$-\infty < t_k < t_{k+1} < T$ for all $k \in \N$
and 
$t_k \to T$ for $k \to \infty$, 
 the solution $\psi \in  L^2(0,T;H^1(0,1)) \cap H^1(0,T;H^1(0,1)^*)$ of 
\begin{equation}
\label{eq:backheat}
\begin{aligned}
		-\partial_t \psi - \partial_x^2 \psi   &= 0 &&\hspace{-0.15cm}\text{in } (0,T) \times (0,1), \\
		\partial_x \psi   &= 0 &&\hspace{-0.15cm}\text{on } (0,T) \times \{0 \}, \\
		\partial_x \psi   &= 0 &&\hspace{-0.15cm}\text{on } (0,T) \times \{1\}, \\
		\psi  &= w    &&\hspace{-0.15cm}\text{on } \{T\} \times (0,1),
\end{aligned}
\end{equation}
has a trace $\psi(\cdot, 1)$ that is analytic 
on $[0, T)$, and, for all $t_k$ satisfying  $t_k \in (0,T)$, it holds $\psi(t_k, 1) > 0$ 
if $k$ is odd and $\psi(t_k, 1) < 0$ if $k$ is even. 
\end{lemma}

\begin{proof}
Recall that the addition formula for the cosine implies
\begin{equation}
\label{eq:cos_ortho}
\begin{aligned}
\int_0^1 \cos(m \pi x ) \cos(n \pi x) \dd x
=
\begin{cases}
0 & \text{ if } n \neq m,
\\[0.2cm]
\displaystyle
\frac12  & \text{ if } n=m,
\end{cases}
\qquad 
\forall m,n \in \N.
\end{aligned}
\end{equation}
In view of 
 \cref{th:power_series}\ref{th:power_series:ii},
 this yields $w \in L^2(0,1)$ and 
\[
\|w\|_{L^2(0,1)}^2 = 
\sum_{m=1}^{\infty} \frac{1}{2} \left | \beta_m \right |^2
=
\frac{1}{2}  \sum_{m=1}^{\infty} \left (  \frac1m \right )^2
= 
\frac{\pi^2}{12}
> 0
\]
 by Parseval's identity. 
 That  $-\infty < t_k < t_{k+1} < T$ holds for all $k$
and that $t_k \to T$ for $k \to \infty$ follows 
immediately from \cref{th:power_series}\ref{th:power_series:i} and \eqref{eq:tk-def}.
Consider now the problem \eqref{eq:backheat} with terminal datum $w \in L^2(0,1)$. 
From the results in \cite[Section~3.2.2]{Troeltzsch2010}, we obtain that the weak 
solution $\psi \in  L^2(0,T;H^1(0,1)) \cap H^1(0,T;H^1(0,1)^*)$
of this reversed heat equation is given by 
\[
	\psi(t, x)
	=
	\int_0^1
	G(x, \xi, T - t) w(\xi) \dd \xi,
\]
where $G$ denotes Green's function for the heat equation with homogeneous
Neumann boundary conditions, i.e., 
\[
G(x, \xi, s) :=
1 + 2 \sum_{n=1}^\infty \cos(n \pi x)
\cos(n \pi \xi) \e^{-n^2 \pi^2 s}.
\]
If we use \eqref{eq:w-def} and again invoke \eqref{eq:cos_ortho}, 
then this provides us with 
\begin{equation}
\label{eq:randomeq36363}
\begin{aligned}
\psi(t, 1) &= \int_0^1 G(1, \xi, T-t) w(\xi) \dd \xi
\\
&=
\int_0^1 \left (
1 + 2 \sum_{n=1}^\infty (-1)^n
\cos(n\pi \xi) \e^{-n^2 \pi^2 (T - t)}
\right )
\left ( 
\sum_{m=1}^\infty (-1)^m\beta_m \cos(m \pi \xi)
\right ) \dd \xi
\\
 &= 
 \sum_{m=1}^\infty   \beta_m  \e^{-m^2 \pi^2 (T - t)}
 \\
&= P \left (  \e^{\pi^2 (t - T)} \right )
 \qquad \forall t \in [0,T),
\end{aligned}
\end{equation}
where $P$ denotes the power series 
from \cref{th:power_series}
associated with the exponents $\alpha_m = m^2$, $m \in \N$.
The analyticity of the trace $\psi(\cdot, 1)$ 
on $[0, T)$ follows immediately from the formula in \eqref{eq:randomeq36363}. 
As $P$ satisfies $P( z_k) > 0$ for all $k \in \N$ that are odd
and $P( z_k) < 0$ for all $k \in \N$ that are even by \cref{th:power_series}\ref{th:power_series:iii}, 
we further obtain from \eqref{eq:tk-def} that $\psi(\cdot, 1)$
switches its sign as asserted. This completes the proof.  
\end{proof}

\begin{remark}
If, in the situation of \cref{lemma:oscadjoint}, the series in \eqref{eq:w-def} is replaced by a finite sum 
starting at $m=1$ and ending at $m = q_L$ with $q_L$, $L \in \N$, 
being an element of the sequence $\{q_k\}_{k \in \N}$ generated 
by \cref{algo:AlgoHarmonic}, then the solution of \eqref{eq:backheat}
satisfies $\psi(t, 1) = P_L( \exp(\pi^2 (t - T)))$ for all $t \in [0, T)$,
where  $P_L$ denotes the polynomial in \eqref{eq:PdefPl}.
In view of \cref{th:power_series}\ref{th:power_series:iv},
the latter implies that $\psi(\cdot, 1)$ changes its sign a finite number of 
times on $[0,T)$ if $L$ is large enough
and reproduces the oscillatory behavior observed in 
the infinite summation case in the limit $L \to \infty$. 
Since the polynomial case has been discussed
in detail in \cite{troeltzsch2023bang}
(also with regard to the question of whether it is possible to prescribe 
where changes of the sign occur)
and since we are mainly interested in the existence of chattering optimal controls, 
we focus on the case of an infinite summation bound in 
\eqref{eq:w-def} in our analysis. 
\end{remark}
\smallskip

The function $w$ that arises from 
the construction in \cref{lemma:oscadjoint} for $z_1 = 0.5$
and the sequences $\{\beta_m\}_{m \in \N}$ and
 $\{z _k\}_{k \in \N}$ in \cref{tab:1} 
can be seen in \Cref{fig:2}.

\begin{figure}[htp]
\begin{center}
  \includegraphics[width=11cm]{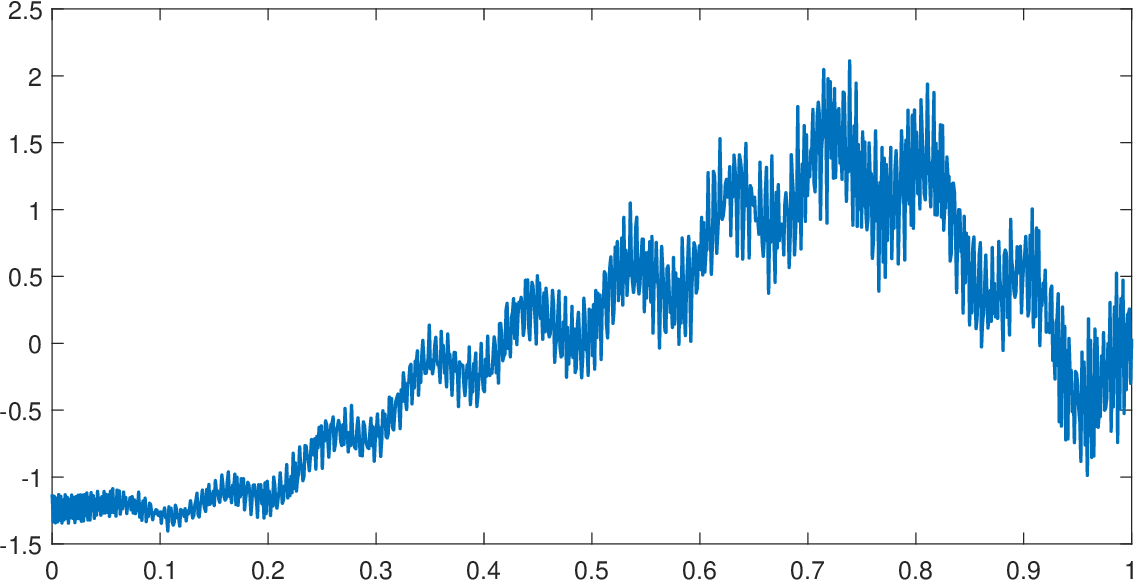}\\
  \caption{The function $w$ in \cref{lemma:oscadjoint} for $z_1 = 0.5$
and the sequences $\{\beta_m\}_{m \in \N}$ and
 $\{z _k\}_{k \in \N}$ seen in \cref{tab:1}.
  }\label{fig:2}
  \end{center}
\end{figure}

We are now in the position to prove the main result of this work, 
namely, 
that chattering optimal controls are indeed possible for the problem \eqref{eq:P}.
\smallskip

\begin{theorem}[{Existence of chattering optimal controls}]
\label{th:main_bang}
Let $T>0$ be fixed. Then there exists 
a desired state $y_d \in L^2(0,1) $ such that 
\eqref{eq:P} possesses a positive 
optimal objective function value, 
such that the analytic function 
$\bar p(\cdot, 1)\colon [0,T) \to \R$
in \cref{prop:bang_bang_principle}
changes infinitely often 
from positive to negative values and back,
and such that the right-continuous representative $\bar u\colon [0,T) \to \{-1,1\}$ 
of the optimal control 
in \cref{prop:bang_bang_principle} 
switches its function value an infinite number of times.
\end{theorem}

\begin{proof}
Let $w$ be constructed as in \cref{lemma:oscadjoint}
for some $z_1 \in (0,1)$
and let $\psi$ be the associated solution of \eqref{eq:backheat}.
We define $\bar u := - \sgn(\psi(\cdot, 1)) \in L^\infty(0,T)$ and denote 
the solution of the state equation of \eqref{eq:P} with boundary datum $\bar u$ by $\bar y$.
We further define $y_d := \bar y(T, \cdot) - w \in L^2(0,1)$. 
Then it holds by construction that $(\bar u, \bar y)$ satisfies the system 
\eqref{eq:optsysP} with adjoint state $\bar p := \psi$ and, consequently, 
that $(\bar u, \bar y)$ solves \eqref{eq:P} with desired state $y_d$. 
As $w = \bar y(T, \cdot) - y_d \in L^2(0,1) \setminus \{0\} $,
it follows that the optimal objective function value of \eqref{eq:P} with datum $y_d$ is positive.
From the properties of $\psi = \bar p$ in \cref{lemma:oscadjoint} 
and the definition of $\bar u$, we further immediately obtain that 
$\bar u$ and $\bar p(\cdot, 1)$ switch function values as asserted.
This completes the proof. 
\end{proof}

We remark 
that we expect that the approach developed in this section
can also be used to construct chattering solutions 
for optimal control problems other than \eqref{eq:P}. 
(\Cref{th:power_series}, which is at the heart of our analysis, is, after all, completely 
independent of \eqref{eq:P}.) We leave possible extensions for future research.

\section*{Acknowledgments}
The author would like to thank
Fredi Tr\"oltzsch for valuable feedback on an early draft of this paper and for pointing out relevant references.









\bibliographystyle{alpha}
\bibliography{references}


%

\medskip
Received xxxx 20xx; revised xxxx 20xx; early access xxxx 20xx.
\medskip

\end{document}